\documentclass[11pt, a4paper, oneside]{article}
\usepackage{amsmath, amssymb, amsthm, amsfonts, amsxtra, latexsym, amscd,
pb-diagram,  graphics, hyperref}
\usepackage [all]{xy}

\usepackage{hyperref}

\theoremstyle{plain}

\theoremstyle{plain}
\newtheorem{theorem}{Theorem}

\newtheorem{corollary}[theorem]{Corollary}
\newtheorem{lemma}[theorem]{Lemma}
\newtheorem{proposition}[theorem]{Proposition}

\newcommand{\ri}{\rightarrow }
\newcommand{\Fa}{\breve{F}}
\newcommand{\Lo}{\widehat{L}}

\newcommand{\A}{\mathcal{A}}

\newcommand{\Z}{\mathbb{Z}}

\newcommand{\Lh}{\frak{L}}
\newcommand{\Rh}{\frak{R}}
\DeclareMathOperator{\Aut}{Aut}
\newcommand{\Fm}{\widetilde{F}}
\def\tx{\otimes}
\def\ts{\oplus}
\def\Fn{\widetilde F}

\begin{document}


\title{Cohomological Classification of Ann-categories}

\author{Nguyen Tien Quang}
\maketitle





\begin{abstract}
An  Ann-category is a categorification of  rings. Regular
Ann-categories were classified by Shukla   cohomology of algebras.
In this paper, we state and prove the precise theorem on
classification of Ann-categories in the general case based on
Mac Lane cohomology of rings.

\end{abstract}
{\bf Keywords:} Ann-category, Ann-functor,  categorical ring, Mac Lane cohomology,  Shukla cohomology\\
{\bf 2010 Mathematics Subject Classification:}  18D10, 16E40



\section{Introduction }
Categories with monoidal structures $\oplus, \otimes$ (or {\it
categories with distributivity constraints})
 were originally considered by   Laplaza in \cite {La1}.   Kapranov and
 Voevodsky  \cite {KV}
  omitted conditions related to
 the commutativity constraint with respect to $\otimes$ in the axioms of Laplaza
  and called these categories {\it ring categories}.

In an alternative approach,  monoidal categories can be ``refined''
to become {\it categories with group structure} if the objects are
all invertible (see \cite {La2}, \cite {SR}). When the underlying
category
  is a {\it groupoid} (that is,  every morphism is
  an isomorphism), we obtain the notion of  {\it monoidal category group-like}  \cite {FW},
   or   {\it Gr-category}  \cite {Si}.
  These categories can be classified by the cohomology group $H^{3}(\Pi, A)$ of groups.

 In 1987,  Quang  \cite {Q1} introduced the notion of  {\it Ann-category}  which is a
categorification of
   rings. Ann-categories are  symmetric Gr-categories (or   Picard categories)
  equipped with a monoidal structure $\otimes$.
 Since all objects are invertible and all morphisms are isomorphisms,
  the  axioms of an Ann-category are much fewer than those of
  a ring category (see \cite {PQT}).
The first two invariants of an Ann-category $\A$ are the ring
$R=\pi_0\A$ of isomorphism classes of the objects in $\A$ and the $
R$-bimodule $M=\pi_1\A = \mathrm{Aut}_{\A}(0).$ Via the structure
transport, we can construct an Ann-category of  type $(R,M)$ which
is Ann-equivalent to $\A.$ A family of constraints of $\A$ induces a
5-tuple of  functions $(\xi, \alpha, \lambda, \rho:R^3\rightarrow M,
\eta:R^2\rightarrow M)$ satisfying  certain relations. This 5-tuple
is called a {\it structure} of an Ann-category of  type $(R, M).$
Our purpose is to classify these categories by an appropriate
cohomology group.

First we deal with the manifold of {\it regular} Ann-categories
(they  satisfy $c_{A,A}= id$ for all objects $A$) which arises from
the ring extension problem. In \cite{Q2}, these categories were
classified by the cohomology group $H^{3}_{Sh}$ of the ring $R$
(regarded as an $\Z$-algebra) in the sense of Shukla \cite{Sh}(that
we misleadingly call Mac Lane-Shukla). This result shows the
relation between the notion of regular Ann-category and
the theory of Shukla cohomology. 
Note that  the structure $(\xi,\eta,\alpha,\lambda,\rho)$ of a
regular Ann-category has an extra condition $\eta(x,x)=0$ for the
symmetry constraint. This condition is similar to the requirement
$f\left(\begin{array}{cccc}
0&x\\
x&0\end{array}\right)=0$ so that a 3-cocycle $f$ of Mac Lane
cohomology has a realization \cite{ML}.

In 2007, Jibladze and  Pirashvili \cite{JP} introduced  the notion
of {\it categorical ring} as a slightly modified
 version of the notion of  Ann-category and classified categorical rings
  by the cohomology group $H^{3}_{MaL}(R,M).$ The
 condition $(Ann-1)$ and the compatibility of $\otimes$ with associativity and commutativity constraints with respect to $\oplus$
  are replaced by the compatibility of $\otimes$ with the ``associativity - commutativity'' constraint.
We prove in \cite{QHT} that
   the manifold of all Ann-categories is a subset of the manifold of all categorical
   rings. We also show that there exists a serious gap in the proof of Proposition   2.3 \cite{JP}.
 The authors of \cite{JP} did not prove the existence of the isomorphisms
  $$A\otimes 0\ri 0,\quad 0\otimes A\ri 0,$$
so that the distributivity constraints induce the $\otimes$-functors
which are compatible with the unit constraints. Thus, it can not be
deduced the $\pi_0\mathcal A$-bimodule structure of the abelian
group $\pi_1\mathcal A$ from axioms of a categorical ring, and
therefore results on cohomological classification of categorical
rings can not be stated precisely. In the appendix, we give an
example of a  {\it categorical ring} which is not an {\it
Ann-category}, and prove that the classification theorem in
\cite{JP} is wrong.





A main result of this paper is the cohomological classification
theorem for Ann-categories (Theorem 12) in the general case. It is
not only a continuation of the results in \cite{Q2} and in
\cite{QH1}, but it also gives a new interpretation of
low-dimensional Mac Lane cohomology groups.

After this introductory Section 1,  Section 2 is devoted to
 recalling  some
well-known results: i) the construction of an Ann-category of type
$(R,M)$ which is the reduced Ann-category of an arbitrary one and
the determination of a {\it structure} on such an Ann-category of
type $(R,M)$; ii) the Mac Lane cohomology and the obstruction theory
of Ann-functors. In section 3 we prove that there is a bijection
$$ \mathrm{Struct}[R,M]\leftrightarrow H^{3}_{MacL}(R,M)$$
between the set of cohomology classes of structures on $(R,M)$ and
the Mac Lane cohomology group
 of the ring $R$ with coefficients in the $R$-bimodule $M,$ and therefore we obtain the precise theorem on
 classification of  Ann-categories and Ann-functors. 

  For short, we sometimes
 write $AB$ or $A.B$ instead of $A\tx B.$


\section {Ann-categories of  type
$(R,M)$}

Let us recall some necessary  concepts and facts in this section
from \cite{Q1, Q2}.

A monoidal category is called a {\it Gr-category} (or a {\it
categorical group}) if every object is invertible and the background
category is a groupoid. A {\it Picard } category (or a {\it
symmetric} categorical group) is a Gr-category equipped with a
symmetry constraint which is compatible with associativity
constraint.

\subsection{Ann-categories and Ann-functors}
\noindent{\bf Definition.} An {\it Ann-category} consists of

\noindent $\mathrm{i)}$ a category  $\A$ together with two
bifunctors $\ts, \tx :\A \times \A \rightarrow \A$;

\noindent $\mathrm{ii)}$ a fixed object $0 \in \A$ together with
natural isomorphisms  ${\bf a_+, c,g,d}$ such that $(\A, \ts,{\bf
a_+,c},(0,{\bf g,d}))$ is a Picard category;

\noindent $\mathrm{iii)}$ a fixed object $1 \in \A$ together with
natural isomorphisms  ${\bf a,l,r}$ such that $(\A,\tx,{\bf
a},(1,{\bf l,r}))$ is a monoidal category;

\noindent $\mathrm{iv)}$ natural isomorphisms $\Lh,\Rh$ given by
\[\begin{array}{ccccc}
\Lh_{A,X,Y}: & A\tx (X\ts Y) &\longrightarrow & (A\tx X)\ts (A\tx Y),\\
\Rh_{X,Y,A}: & (X\ts Y)\tx A &\longrightarrow & (X\tx A)\ts (Y\tx A)
\end{array}\]
such that the following conditions hold:

\noindent (Ann - 1) for  $A\in \A$, the pairs $(L^A, \Breve L^A),
(R^A, \Breve R^A)$ defined by
\[\begin{array}{lclclclcc}
L^A & = & A\tx - & \qquad\qquad & R^A & = & -\tx A \\
\breve L^A_{X,Y} & = & \Lh_{A,X,Y} & \qquad\qquad & \breve R^A_{X,Y}
& = & \Rh_{X,Y,A}
\end{array}\]
are $\ts$-functors which are compatible with ${\bf a_+}$ and ${\bf
c}$;

\noindent (Ann - 2) for all $A,B,X,Y \in \A,$ the following diagrams
commute  {\scriptsize
\[\begin{diagram}
\node{(AB)(X\ts Y)}\arrow{s,l}{\Breve L^{AB}}\node{A(B(X\ts
Y))}\arrow{w,t} {{\bf a}_{A,B,X\ts Y}}\arrow{e,t}{id_A\tx \Breve
L^B}\node{A(BX \ts BY)}\arrow{s,r}
{\Breve L^A}\\
\node{(AB)X\ts (AB)Y}\node[2]{A(BX)\ts A(BY)}\arrow[2]{w,t}{{\bf
a}_{A,B,X}\ts {\bf a}_{A,B,Y}}
\end{diagram}\]

\[\begin{diagram}
\node{(X\ts Y)(BA)}\arrow{s,l}{\Breve R^{BA}}\arrow{e,t}{{\bf
a}_{X\ts Y,B,A}} \node{((X\ts Y)B)A}\arrow{e,t}{\Breve R^B\tx id_A}
\node{(XB \ts YB)A}\arrow{s,r}{\Breve R^A}\\
\node{X(BA)\ts Y(BA)}\arrow[2]{e,t}{{\bf a}_{X,B,A}\ts {\bf
a}_{Y,B,A}}\node[2]{(XB)A \ts (YB)A}
\end{diagram}\]
\[\begin{diagram}
\node{(A(X\ts Y)B}\arrow{s,l}{\Breve L^A \tx id_B} \node{A((X\ts
Y)B)}\arrow{w,t}{{\bf a}_{A,X\ts Y,B}}\arrow{e,t}{id_A\tx \breve
R^B}
\node{A(XB \ts YB)}\arrow{s,r}{\Breve L^A}\\
\node{(AX\ts AY)B}\arrow{e,t}{\breve R^B}\node{(AX)B \ts (AY)B}
\node{A(XB) \ts A(YB)}\arrow{w,t}{{\bf a}\ts {\bf a}}
\end{diagram}\]
\[\begin{diagram}
\node{(A\ts B)X\ts (A\ts B)Y}\arrow{s,l}{\Breve R^X \ts \breve R^Y}
\node{(A\ts B)(X\ts Y)}\arrow{w,t}{\breve L^{A\ts
B}}\arrow{e,t}{\breve R^{X\ts Y}}
\node{A(X \ts Y)\ts B(X\ts Y)}\arrow{s,r}{\Breve L^A \ts\Breve L^B }\\
\node{(AX\ts BX)\ts (AY\ts BY)}\arrow[2]{e,t}{{\bf v}}\node[2]{(AX
\ts AY)\ts (BX \ts BY)}
\end{diagram}\]}
where ${\bf v} = {\bf v}_{U,V,Z,T}:(U\ts V)\ts (Z\ts T)
\longrightarrow (U\ts Z)\ts (V\ts T)$ is a unique  morphism
constructed from $\oplus, {\bf a_+, c}, id$ of the symmetric
monoidal category $(\A,\ts)$;

\noindent (Ann - 3) for the unit $1 \in \A$ of the operation $\tx$,
the following diagrams commute {\scriptsize
\[\begin{diagram}
\node{1(X\ts Y)} \arrow[2]{e,t}{\breve L^{1}}\arrow{se,b}{{\bf
l}_{X\ts Y}} \node[2]{1 X\ts 1 Y}\arrow{sw,b}{{\bf l}_X\ts {\bf
l}_Y} \node{(X\ts Y)1}\arrow[2]{e,t}{\breve R^{1}}\arrow{se,b}{{\bf
r}_{X\ts Y}}
\node[2]{X1 \ts Y1}\arrow{sw,b}{{\bf r}_X\ts {\bf r}_Y}\\
\node[2]{X\ts Y} \node[3]{X\ts Y.}
\end{diagram}\]}

Since each of pairs  $(L^A, \hat L^A)$, $(R^A, \hat R^A)$ is an
$\oplus$-functor which is compatible with the associativity
constraint in the Picard category  $\A$, it is also compatible with
the unit constraint $(0,{\bf g,d}),$ so we obtain the following
result.

\begin{lemma}
In an Ann-category $\mathcal A$ there exist uniquely isomorphisms
\[\widehat{L}^A\;:\;A\otimes 0\longrightarrow 0\;,\widehat{R}^A\;:\;0\otimes A\longrightarrow 0\] such that
the following diagrams commute {\scriptsize
\[\begin{diagram}
\node{AX}\node{A(0\ts X)}\arrow{w,t}{L^A({\bf g})}\arrow{s,r}{\breve L^A}
\node{AX}\node{A(X\ts 0)}\arrow{w,t}{L^A({\bf d})}\arrow{s,r}{\breve L^A}\\
\node{0\ts AX}\arrow{n,l}{{\bf g}}\node{A0\ts AX}\arrow{w,t}{\hat L^A\ts id}
\node{AX\ts 0}\arrow{n,l}{{\bf d}}\node{AX\ts A0}\arrow{w,t}{id\ts \hat L^A}
\end{diagram}\]
\[\begin{diagram}
\node{XA}\node{(0\ts X)A}\arrow{w,t}{R^A({\bf g})}\arrow{s,r}{\breve R^A}
\node{XA}\node{(X\ts 0)A}\arrow{w,t}{R^A({\bf d})}\arrow{s,r}{\breve R^A}\\
\node{0\ts XA}\arrow{n,l}{{\bf g}}\node{0A\ts XA}\arrow{w,t}{\hat
R^A\ts id} \node{XA\ts 0}\arrow{n,l}{{\bf d}}\node{XA\ts
0A.}\arrow{w,t}{id\ts \hat R^A}
\end{diagram}\]}
\end{lemma}

It is easy to see that if $(F,\Fa, \widehat{F} ):(\A,\oplus)\ri
(\A',\oplus)$ is a monoidal functor between two Gr-categories, then
the canonical isomorphism $\widehat{F}:F0\ri 0'$ can be deduced from
the others. Thus, we state the following definition.

\noindent{\bf Definition.} Let $\A$ and $\A'$ be Ann-categories. An
{\it Ann-functor} $(F,\breve F, \Fn, F_\ast):\A\ri \A'$ consists of
a functor $F:\A\rightarrow \A'$, natural isomorphisms
$$\breve F_{X,Y}:F(X\oplus Y)\rightarrow F(X)\oplus F(Y),\ \widetilde{F} _{X,Y}:F(X\otimes
 Y)\rightarrow F(X)\otimes F(Y),$$
 and an isomorphism $F_\ast:F(1)\rightarrow 1'$ such that $(F,\breve F)$ is a symmetric monoidal functor
with respect to  the operation  $\ts$, $(F,\Fn, F_\ast)$ is a
monoidal functor with respect to  the operation $\tx$, and
$(F,\breve F,\Fn)$ satisfies two following commutative diagrams
{\scriptsize
\[\begin{diagram}
\node{F(X(Y\ts Z))}\arrow{s,l}{F(\Lh)}
\arrow{e,t}{\Fn}\node{FX.F(Y\ts Z)}\arrow{e,t}{id
\tx\Fn}\node{FX(FY\ts FZ)}
\arrow{s,r}{\Lh'}\\
\node{F(XY\ts XZ)}\arrow{e,t}{\breve F}\node{F(XY)\ts F(XZ)}\arrow{e,t}{\Fn\ts\Fn}\node{FX.FY\ts FX.FZ}
\end{diagram}\]
\[\begin{diagram}
\node{F((X\ts Y)Z)}\arrow{s,l}{F(\Rh)}\arrow{e,t}{\Fn}\node{F(X\ts Y).FZ}\arrow{e,t}{\breve F\tx id}\node{(FX\ts FY)FZ}\arrow{s,r}{\Rh'}\\
\node{F(XZ\ts YZ)}\arrow{e,t}{\breve F}\node{F(XZ)\ts F(YZ)} \arrow{e,t}{\Fn\ts\Fn}\node{FX.FZ\ts FY.FZ.}
\end{diagram}\]}

These diagrams are  called the {\it compatibility} of the functor
$F$ with the distributivity constraints.

\indent An {\it Ann-morphism} (or a {\it homotopy})
$$\theta : (F, \Fa, \widetilde{F}, F_{\ast})\rightarrow (F', \breve{F'}, \widetilde{F}', F'_{\ast})$$
between Ann-functors is an  $\oplus$-morphism, as well as an
$\otimes$-morphism.

 If there exists an Ann-functor $(F',
\breve{F'}, \widetilde{F}', F'_{\ast}):\A'\rightarrow \A$ and
Ann-morphisms $F'F\stackrel{\sim}{\rightarrow} id_{\A}, \
FF'\stackrel{\sim}{\rightarrow} id_{\A'}$, we say that
$(F,\Fa,\widetilde{F}, F_{\ast})$ is an {\it Ann-equivalence}, and
$\A$, $\A'$ are {\it Ann-equivalent}.

 It can be proved that each
Ann-functor is an Ann-equivalence if and only if $F$ is a
categorical equivalence.

\begin{lemma}\label{lem31}
Any Ann-functor $F=(F, \Fa, \widetilde{F}, F_{\ast}): \A\rightarrow
\A'$ is homotopic to an Ann-functor $F'=(F', \Fa', \widetilde{F}',
F_{\ast}'),$ where $F'{0}={0}', \widehat{F}'0=id_{{ 0}'},$ and
$F'{1}={1}', F'_{\ast}=id_{{ 1}'}$.
\end{lemma}
\begin{proof}
Consider a family of isomorphisms in $\A'$:
$$\theta_X=\begin{cases}\begin{aligned}id_{FX} &\ \text{if}\ X\not= { 0},\;X\not= { 1},\\
 \widehat{F}\;\;\;\;\;&\ \text{if}\ X={ 0}, \\ F_{\ast}\;\;\;\;\;&\ \text{if}\ X={ 1},\end{aligned}\end{cases}$$
for $X\in \A$. Then, 
the Ann-functor
 $F'$ can be constructed in a unique way such that $\theta: F\rightarrow F'$ becomes a homotopy. Namely,
\begin{eqnarray*}
F'X&=&\begin{cases} \begin{aligned}FX &\ \text{if}\ X\not= 0, X\not= 1,\\ 0'\;\;\;& \ \text{if}\  X=0,
 \\ 1'\;\;\;& \ \text{if}\  X=1,\end{aligned}\end{cases}\\
 F'(f: X\rightarrow Y)&=&\theta_Y F(f)(\theta_X)^{-1}: F'X\rightarrow F'Y,\\
\Fa'_{X,Y}&=&(\theta_X\oplus \theta_Y)\Fa_{X,Y}\theta^{-1}_{X\oplus Y},\\ \widetilde{F'}_{X,Y}&=&(\theta_X\otimes\theta_Y)\widetilde{F}_{X,Y}\theta^{-1}_{XY},\\
 \widehat{F'}&=& \widehat{F}\theta^{-1}_0=id_{0'},\ F_{\ast}'= F_{\ast}\theta^{-1}_1=id_{1'}.
\end{eqnarray*}
\
\end{proof}
Based on Lemma \ref{lem31}, we refer to $(F, \Fa, \widetilde{F})$ as
an Ann-functor.
\subsection{Reduced Ann-categories}

For an Ann-category $\A,$ the set $R=\pi_0\mathcal A$ of isomorphism
classes of the objects in $\A$ is a ring where the operations
$+,\times$ are induced by $\ts,\tx$ on $\mathcal A,$ and
$M=\pi_1\mathcal A = \Aut(0)$ is an abelian group where the
operation,  denoted by +, is just the composition. Moreover,
$M=\pi_1\mathcal A$ is an $R$-bimodule with the actions
\begin{align}
sa=\lambda_X(a), \quad as=\rho_X(a),\nonumber
\end{align}
where $X\in s, s\in\pi_0\A, a\in\pi_1\A$ and $\lambda_X,\rho_X$
satisfy the commutative diagrams
\[
\begin{diagram}
\node{X.0}\arrow{e,t}{\hat L^X}\arrow{s,l}{id\tx
a}\node{0}\arrow{s,r}{\lambda_X(a)} \node[2]{0.X}\arrow{e,t}{\hat
R^X}\arrow{s,l}{a\tx id}\node{0}\arrow{s,r}{\rho_X(a)}
\\
\node{X.0}\arrow{e,t}{\hat L^X}\node{0,}
\node[2]{0.X}\arrow{e,t}{\hat R^X}\node{0.}
\end{diagram}
\]

We recall briefly
some main facts of the construction of the reduced
Ann-category $S_{\A}$
 of   $\A$ via the structure transport  (for details, see
\cite{Q2}).
The objects of $S_{\A}$ are the elements of the ring
$\pi_0\A$. A morphism is an
 automorphism $(s,a): s\rightarrow s,\ s\in \pi_0\A, a\in
\pi_1\A$. The composition of morphisms is given by
\[(s,a)\circ (s,b)=(s,a+b).\]

For each  $ s\in \pi_0\A$, choose an object $X_s\in\mathcal A$ such
that $X_0=0, X_1=1$,
  and choose an isomorphism $i_X: X\rightarrow X_s$ such that $i_{X_s}=id_{X_s}$.
We obtain two functors \begin{equation}\label{gh}
\begin{cases}
G:\mathcal A\ri S_{\mathcal A}\\
G(X)=[ X]=s\\
G(X \stackrel {f}{\ri}Y)=(s,\gamma_{X_s}^{-1}(i_Yfi_X^{-1})),
\end{cases}\qquad\qquad
\begin{cases}
H: S_{\mathcal A}\ri  \mathcal A\\
H(s)=X_s\\
H(s,u)=\gamma_{X_s}(u).
\end{cases}
\end{equation}
\noindent  where $X,Y\in s$ and $f: X\rightarrow Y$, and $\gamma_X$
is a map defined by the following commutative diagram
\[
\begin{diagram}\label{d9}
\node{X}\arrow{e,t}{\gamma_X(a)}
\node{X}\\
\node{0\ts X}\arrow{n,l}{{\bf g}_{_{X}}}\arrow{e,t}{a\ts id}
\node{0\ts X}\arrow{n,r}{{\bf g}_X}
\end{diagram}\]
\centerline{\scriptsize Diagram 1} \vspace{0.1pt}

\noindent The operations on $S_{\A}$ are defined by
\begin{eqnarray*}
s\ts t&=&G(H(s)\ts H(t))=s+t,\\
(s,a)\ts (t,b)&=&G(H(s,a)\ts H(t,b))=(s+t,a+b),\\
s\tx t&=& G(H(s)\tx H(t))=st,\\
(s,a)\tx (t,b)&=&G(H(s,a)\tx H(t,b))=(st, sb+at),
\end{eqnarray*}
where $s,t\in \pi_0\A$, $a,b\in \pi_1\A$. Obviously, these
operations do not depend on the choice of the set of representatives
$(X_s, i_X).$

The constraints in $S_{\A}$ are defined by those in $\A$ by means of
the notion of {\it stick}. A {\it stick} in $\A$ is a set of
representatives $(X_s, i_X)$ such that
\begin{eqnarray*}
i_{0\oplus X_t}={\bf g}_{X_t},\qquad  i_{X_s\oplus 0}={\bf d}_{X_s},&\\
i_{1\otimes X_t}={\bf l}_{X_t}, \qquad i_{X_s\otimes 1}={\bf r}_{X_s},& i_{0\otimes X_t}=\widehat{R}^{X_t},\quad i_{X_s\otimes 0}=\Lo^{X_s}.
\end{eqnarray*}

 The unit constraints for two operations $\ts,\tx$ in $S_{\A}$ are $(0, id,id)$ and $(1,id, id)$, respectively.
 The functor $H$ and isomorphisms
\begin{align}
\breve{H}=i^{-1}_{X_s\oplus X_t},\ \widetilde{H}=i^{-1}_{X_s\otimes
X_t}.\label{eq2}
\end{align}
transport the constraints  ${\bf a_+, c, a}, \frak{L}, \frak{R}$ of
$\A$ to those $\xi, \eta, \alpha, \lambda, \rho$ of $S_{\A}$. Then,
the category
$$(S_{\A}, \xi,\eta, (0, id, id), \alpha, (1, id, id), \lambda, \rho)$$
is an Ann-category which is equivalent to $\A$ by the
Ann-equivalence $(H, \breve{H}, \widetilde{H}):S_{\A}\ri \A$.
Besides, the functor $G: \A\rightarrow S_{\A}$ together with
isomorphisms
\begin{align}
\breve{G}_{X,Y}=G(i_X\ts i_Y),\ \widetilde{G}_{X,Y}=G(i_X\tx i_Y)
\label{eq4}
\end{align}
is also an Ann-equivalence.
 We refer to $S_{\A}$ as an Ann-category of {\it type} $(R,M),$ called a
 {\it reduction} of $\mathcal A$. We also call
   $(H, \breve H,\widetilde{H})$ and $(G, \breve G,\widetilde{G})$  {\it canonical
 } Ann-equivalences, the family of constraints $h=(\xi, \eta, \alpha, \lambda, \rho)$ of
$S_{\A}$   a {\it structure} of the Ann-category of type $(R, M)$,
or simply a {\it structure on $(R,M)$}.

The following result follows from the axioms of an Ann-category.
\begin{theorem} [{Theorem 3.1 \cite{Q2}}]\label{dl3}
In the reduced Ann-category $S_{\A} $ of an Ann-category $\mathcal
A,$ the structure $(\xi, \eta, \alpha, \lambda, \rho)$ consists of
functions with values in $\pi_1\mathcal A$ such that for any
$x,y,z,t\in\pi_0\mathcal A$, the following conditions hold:
\begin{eqnarray*}
A_1.&
\xi(y,z,t) - \xi(x+y,z,t) +\xi(x,y+z,t) - \xi(x,y,z+t) + \xi(x,y,z) = 0,\\
A_2.&\xi(x,y,z) - \xi(x,z,y) +\xi(z,x,y) + \eta(x+y,z) - \eta(x,z) - \eta(y,z) = 0,\\
A_3.&\eta(x,y) + \eta(y,x) = 0,\\
A_4.&x\eta(y,z) - \eta(xy,xz) = \lambda(x,y,z) - \lambda(x,z,y),\\
A_5.&\eta(x,y)z - \eta(xz,yz) =  \rho(x,y,z) - \rho(y,x,z),\\
A_6.&x\xi(y,z,t) - \xi(xy,xz,xt) =\lambda(x,z,t) - \lambda(x,y+z,t) + \lambda(x,y,z+t)\\
& - \lambda(x,y,z),\\
A_7.& \xi(x,y,z)t - \xi(xt,yt,zt) = \rho(y,z,t) - \rho(x+y,z,t)+\rho(x,y+z,t)\\
&-\rho(x,y,z),\\
A_8.&\rho(x,y,z+t) - \rho(x,y,z) -\rho(x,y,t) +\lambda(x,z,t)+\lambda(y,z,t) - \lambda(x+y,z,t) =\\ &\xi(xz+xt,yz,yt)+\xi(xz,xt,yz)-\eta(xt,yz)+\xi(xz+yz,xt,yt)-\xi(xz,yz,xt),\\
A_9.&\alpha(x,y,z+t)-\alpha(x,y,z)-\alpha(x,y,t) = x\lambda(y,z,t) + \lambda(x,yz,yt)\\
&-\lambda(xy,z,t),\\
A_{10}.&\alpha(x,y+z,t)-\alpha(x,y,t)-\alpha(x,z,t) = x\rho(y,z,t) - \rho(xy,xz,t)+\\
&\lambda(x,yt,zt) -\lambda(x,y,z)t,\\
A_{11}.&\alpha(x+y,z,t) - \alpha(x,y,t) -\alpha(y,z,t) = -\rho(x,y,z)t - \rho(xz,yz,t)\\
& + \rho(x,y,zt),\\
A_{12}.&x\alpha(y,z,t) - \alpha(xy,z,t) + \alpha(x,yz,t) - \alpha(x,y,zt) + \alpha(x,y,z)t = 0,
\end{eqnarray*}
\noindent Further, these functions  satisfy normalization
conditions:
\begin{eqnarray*}
\xi(0,y,z)&=&\xi(x,0,z) = \xi(x,y,0) = 0,\\
\alpha(1,y,z)&=&\alpha(x,1,z) = \alpha(x,y,1) = 0,\\
\alpha(0,y,z)&=&\alpha(x,0,z) = \alpha(x,y,0) = 0,\\
\lambda(1,y,z)&=&\lambda(0,y,z) = \lambda(x,0,z) = \lambda(x,y,0) = 0,\\
\rho(x,y,1)&=&\rho(0,y,z) = \rho(x,0,z) = \rho(x,y,0) = 0.
\end{eqnarray*}

\end{theorem}
The induced operations on $S_{\A} $ do not depend on the choice of
sticks. We now investigate  the effect of different choices of the
stick $(X_s, i_X)$ in the induced constraints on $S_{\A} $.
\begin{proposition}\label{pr4}
Let $\mathcal S$ and $\mathcal S'$ be reduced Ann-categories of
$\mathcal A$ corresponding to the sticks $(X_s,i_X)$ and
$(X'_s,i'_X)$, respectively. Then the structures
$(\xi,\eta,\alpha,\lambda,\rho)$ of $\mathcal S$ and
$(\xi',\eta',\alpha',\lambda',\rho')$ of $\mathcal S'$ satisfy the
following relations:
\begin{eqnarray*}
A_{13}.& \xi(x,y,z) - \xi'(x,y,z)  =  \tau(y,z)- \tau(x+y,z) + \tau(x,y+z) - \tau(x,y),\\
A_{14}.& \eta(x,y) - \eta'(y,x)  =  \tau(x,y) - \tau(y,x),\\
A_{15}.& \alpha(x,y,z) - \alpha'(x,y,z)  =  x\nu(y,z) - \nu(xy,z) + \nu(x,yz) - \nu(x,y)z,\\
A_{16}.& \lambda(x,y,z) - \lambda'(x,y,z)  = \nu(x,y+z) - \nu(x,y) - \nu(x,z) + x\tau(y,z) -\tau(xy,xz),\\
A_{17}.& \rho(x,y,z) - \rho'(x,y,z)  = \nu(x+y,z) - \nu(x,z) - \nu(y,z) + \tau(x,y)z - \tau(xz,yz),
\end{eqnarray*}
where $\tau,\nu:(\pi_0\A)^2\rightarrow\pi_1\A$ are the functions satisfying the normalization conditions $\tau(0,y) = \tau(x,0) = 0$ and $\nu(0,y) = \nu(x,0) = \nu(1,y) = \nu(x,1) = 0$.
\end{proposition}
Two structures $(\xi,\eta,\alpha,\lambda,\rho)$ and
$(\xi',\eta',\alpha',\lambda',\rho')$ of  Ann-categories of  type
$(R, M)$  are \emph{cohomologous} if and only if they satisfy the
relations $A_{13}-A_{17}$ in Proposition \ref{pr4}.

Note that two unit constraints of $\ts$ and $\tx$ in an Ann-category
of  type $(R, M)$ are both strict. It is easy to prove the following
lemma.

\begin{lemma} \label{bd5} Two structures $h$ and $h'$ are cohomologous if and only if there exists
 an Ann-functor
 $(F,\breve F,\Fn):(R,M, h) \rightarrow (R,M, h')$, where $ F=id_{(R,M)}.$
\end{lemma}

\subsection{ Mac Lane cohomology groups of rings and obstruction theory}

Let $R$ be a ring and $M$ be an $R$-bimodule. From the definition of
Mac Lane cohomology of rings \cite{ML}, we obtain the description of
 elements in the cohomology group $H^3_{MaL}(R,M).$ 

The group $Z^3_{{MaL}}(R,M)$ of 3-cocycles of $R$ with coefficients
in $M$ consists of the quadruples $(\sigma,\alpha,\lambda,\rho)$ of
the maps:
$$\sigma:R^4\rightarrow M;\ \ \alpha,\lambda,\rho:R^3\rightarrow M
$$
satisfying the following conditions:

 \vspace{5pt}
\noindent $  M_1.\;
x\alpha(y,z,t)-\alpha(xy,z,t)+\alpha(x,yz,t)-\alpha(x,y,zt)+\alpha(x,y,z)t=0,$

\vspace{3pt} \noindent $M_2.\;
-\alpha(x,z,t)+-\alpha(y,z,t)+\alpha(x+y,z,t)+\rho(xz,yz,t)-\rho(x,y,zt)+$\\
\vspace{3pt} $\;\;\;\;\;\;\;\rho(x,y,z)t=0,$

\vspace{3pt} \noindent $M_3.\; -\alpha(x,y,t)-\alpha(x,z,t)+\alpha(x,y+z,t)+x\rho(y,z,t)-\rho(xy,xz,t)-\lambda(x,yt,zt)$\\
\vspace{3pt} $\;\;\;\;\;\;\; +\lambda(x,y,z)t=0,$

\vspace{3pt} \noindent $M_4. \;
\alpha(x,y,z)+\alpha(x,y,t)-\alpha(x,y,z+t)+x\lambda(y,z,t)-\lambda(xy,z,t)+\lambda(x,yz,yt)$\\
\vspace{3pt}$\;\;\;\;\;\;\; =0,$

\vspace{3pt} \noindent $ M_5. \; -\lambda(x,z,t)-\lambda(y,z,t)+\lambda(x+y,z,t)+\rho(x,y,z)+\rho(x,y,t)$\\
\vspace{3pt}$\;\;\;\;\;\;\;-\rho(x,y,z+t)+\sigma(xz,xt,yz,yt)=0,$

\vspace{3pt} \noindent $M_6. \; \lambda(r,x,y)+\lambda(r,z,t)-\lambda(r,x+z,y+t)-\lambda(r,x,z)-\lambda(r,y,t)+
\lambda(r,x+y,z+t)$\\
\vspace{3pt}$\;\;\;\;\;\;\;-r\sigma(x,y,z,t)+\sigma(rx,ry,rz,rt)=0,$

\vspace{3pt} \noindent $M_7. \; -\rho(x,y,r)-\rho(z,t,r)+\rho(x+z,y+t,r)+\rho(x,z,r)+\rho(y,t, r)$\\
\vspace{3pt}$\;\;\;\;\;\;\; -\rho(x+y,z+t,r) -\sigma(xr,yr,zr,tr)+\sigma(x,y,z,t)r=0,$

\vspace{3pt} \noindent $M_8.\; -\sigma(r,s,u,v)-\sigma(x,y,z,t)+\sigma(r+x, s+y, u+z, v+t)$\\
\vspace{3pt}$\;\;\;\;\;\;\;  +\sigma(r,s,x,y)+\sigma(u,v,z,t)-\sigma(r+u, s+v, x+z, y+t)$\\
\vspace{3pt}$\;\;\;\;\;\;\; -\sigma(r,u,x,z)-\sigma(s,v,y,t)+\sigma(r+s,u+v,x+y,z+t)=0.$

\noindent These functions  satisfy normalization
conditions:
\begin{eqnarray*}
\alpha(0,y,z)&=&\alpha(x,0,z)=\alpha(x,y,0)=0,\\
\lambda(0,y,z)&=&\lambda(x,0,z)=\lambda(x,y,0)=0,\\
\rho(0,y,z)&=&\rho(x,0,z)=\rho(x,y,0)=0,\\
\sigma(r,s,0,0)&=&\sigma(0,0,u,v)=\sigma(r,0,u,0)=\sigma(0,s,0,v)=\sigma(r,0,0,v)=0.
\end{eqnarray*}
The 3-cocycle $h=(\sigma,\alpha,\lambda,\rho)$ belongs to the group
$B^3_{{MaL}}(R,M)$ if and only if there exist the functions
$\tau\nu:R^2\rightarrow M$ satisfying:

\vspace{5pt} \noindent $M_9.\;\sigma (x,y,z,t)=\tau(x,y)+\tau(z,t)-\tau(x+z,y+t)-\tau(x,z)-\tau(y,t)$\\
\vspace{3pt}$\;\;\;\;\;\;\; +\tau(x+y,z+t),$

\vspace{3pt} 
\noindent $M_{10}.\;\alpha(x,y,z)=x\nu(y,z)-\nu(xy,z)+\nu (x,yz)-\nu(x,y)z,$

\vspace{3pt} \noindent
$M_{11}.\;\lambda(x,y,z)=\nu(x,y+z)-\nu(x,y)-\nu(x,z)+x\tau(y,z)-\tau(xy,xz),$

\vspace{3pt} \noindent
$M_{12}.\;\rho(x,y,z)=\nu(x+y,z)-\nu(x,z)-\nu(y,z)+\tau(x,y)z-\tau(xz,yz),$

\vspace{3pt}\noindent where $\tau,\nu$ satisfy the normalization
conditions: $\tau(0,y) = \tau(x,0) = 0$ and $\nu(0,y) = \nu(x,0) =
\nu(1,y) = \nu(x,1) = 0.$

The group $Z^2_{MaL}(R,M)$ consists of 2-cochains $g=(\tau,\nu)$ of the ring $R$ with coefficients in the $R$-bimodule $M$ satisfying
\[\partial g=0.\]

\noindent The subgroup $B^2_{MaL}(R,M)\subset Z^2_{MaL}(R,M)$ of
2-coboundaries consists of the pairs $(\tau,\nu)$ such that there
exist the maps $t:R\rightarrow M$ satisfying
$(\tau,\nu)=\partial_{MaL}t$, that is,  

\vspace{0.15cm}
\noindent $M_{13}.\ \ \tau(x,y) =  t(y)-t(x+y)+t(x),$

\vspace{3pt}
\noindent $M_{14}.\  \ \nu(x,y)  =  xt(y) - t(xy)+t(x)y,$

\vspace{3pt} \noindent where $t$ satisfies the normalization
condition, $t(0)=t(1)=0$.

\vspace{0.15cm}
The group $Z^1_{MaL}(R,M)$ consists of 1-cochains $t$ of the ring $R$ with coefficients in the $R$-bimodule $M$ satisfying
\[\partial t=0.\]

\noindent The subgroup  of 1-coboundaries, $B^1_{MaL}(R,M)\subset Z^1_{MaL}(R,M)$, consists of the functions $t$ such that there exists $a\in R$ satisfying $t(x)=ax-xa$.

The quotient group
\[H^i_{MaL}(R,M)=Z^i_{MaL}(R,M)/ B^i_{MaL}(R,M),\ i=1,2,3, \]
is called the  $i^{\text{th}}${\it Mac Lane cohomology group} of the ring $R$ with coefficients in the $R$-bimodule $M$.

\vspace{0.2cm}
Let us now recall from \cite{QH1} some results on Ann-functors. 
Each Ann-functor $(F, \Fa,\widetilde{F}):\A\rightarrow \A'$ induces
one $S_F$ between   their reduced Ann-categories. Throughout this
section, let $\mathcal S$ and $\mathcal S'$ be Ann-categories of
types $(R, M,h)$ and $(R',M',h')$, respectively.

 A functor  $F: \mathcal S\rightarrow \mathcal S'$ is called a functor of {\it type} $(p, q)$ if
$$F(x)=p(x),\ \ F(x,a)=(p(x), q(a)), $$
where $p:R\rightarrow R'$ is a ring homomorphism and $q:M\rightarrow
M'$ is a group homomorphism such that
$$q(xa)=p(x)q(a),\quad x\in R, a\in M.$$
The group $M'$ can be regarded as an $R$-module with the action
$sa'=p(s)a'$, so $q$ is an $R$-bimodule homomorphism. In this case,
we say that $(p,q)$ is a {\it pair of homomorphisms} and that the
function
\begin{equation}
k=q_{\ast}h -p^{\ast}h' \label{8}
\end{equation}
 is   an {\it obstruction} of $F$,
where $p^{\ast}, q_{\ast}$ are canonical homomorphisms,
$$Z^{3}_{MacL}(R, M)\stackrel{q_{\ast}}{\longrightarrow}Z^{3}_{MacL}(R, M')
\stackrel{p^{\ast}}{\longleftarrow} Z^{3}_{MacL}(R', M').$$

\begin{proposition}[Proposition 4.3 \cite{QH1}]\label{s61}
Every Ann-functor $F: \mathcal S\rightarrow \mathcal S'$ is a
functor of type $(p, q).$
\end{proposition}

Keeping in mind  that $\gamma$ is the map defined by Diagram 1,we
state the following proposition.
\begin{proposition}[Proposition 4.1 \cite{QH1}]\label{md22}
Let $\A$ and $\A'$ be Ann-categories. Then every Ann-functor
$(F,\Fa,\widetilde{F}):\A\rightarrow \A'$ induces an Ann-functor
$S_F:S_{\A}\rightarrow S_{\A'}$ of  type $(p, q),$ where
$$p=F_0:\pi_0\A\to \pi_0\A',\ [X] \mapsto [FX],$$
$$q=F_1:\pi_1\A\to\pi_1\A',\ u \mapsto \gamma_{F0}^{-1}(Fu).$$
 Further,

\emph {i)} $F$ is an equivalence if and only if $F_0, F_1$ are
isomorphisms,

\emph {ii)} the Ann-functor  $S_F$ satisfies the commutative diagram
\[
\begin{diagram}
\node{\A}\arrow{e,t}{F}
\node{\A'}\arrow{s,r}{G'}\\
\node{S_{\A}}\arrow{n,l}{H}\arrow{e,t}{S_F} \node{S_{\A',}}
\end{diagram}\]
where $H, G'$ are canonical Ann-equivalences defined by
$\mathrm{(\ref{gh}), (\ref{eq2}), (\ref{eq4})}$.

\end{proposition}

Since $\Fa_{x,y}=(\bullet, \tau(x,y)), $ and
$\widetilde{F}_{x,y}=(\bullet, \nu(x,y)),$
  we call $g_F=(\tau, \nu)$ a pair of functions {\it associated } to $(\Fa, \widetilde{F}),$ and hence
   an Ann-functor $F:\mathcal S\rightarrow \mathcal S'$ can be regarded as a triple $(p, q, g_F).$
    It follows from the compatibility of $F$ with the constraints
    that
\begin{equation}\label{ht9}
q_{\ast}h-p^{\ast}h'=\partial (g_F),
\end{equation}
Moreover, Ann-functors  $(F,g_F)$ and $(F',g_{F'})$  are homotopic
if and only if $F'=F$, that is,  they are of the same type $(p,q)$,
and there is a function $t:R\to M'$ such that
  $g_{F'}=g_F+\partial t.$

We write
 $$\mathrm{Hom}_{(p, q)}^{Ann}[\mathcal S, \mathcal S']$$
 for the set of homotopy classes of Ann-functors of  type $(p,q)$ from $\mathcal S$ to $\mathcal S'.$



\begin{theorem}[Theorem 4.4, 4.5 {\cite{QH1}} ]\label{s63}
The functor  $F:\mathcal S\rightarrow \mathcal S' $ of  type $(p,
q)$ is an Ann-functor if and only if the obstruction $[k]$ vanishes
in $H_{MacL}^3(R, M')$. Then, there exists a bijection
\begin{equation}\label{ht11} \mathrm{Hom}_{(p, q)}^{Ann}[\mathcal S, \mathcal S']\leftrightarrow H^2_{MacL}(R,
M').
\end{equation}
\end{theorem}

\section{Classification of Ann-categories}
\indent In order to prove the main result (Theorem \ref{dl51}) of
the paper,  we first prove that the set of cohomology classes of
structures on $(R,M)$ and the group $H^3_{MaL}(R,M)$ are coincident.
\begin{lemma}\label{bd6}
Each structure of an Ann-category of  type $(R,M)$ induces a
3-cocycle in $Z^3_{{MaL}}(R,M).$
\end{lemma}
\begin{proof}
Let $h=(\xi,\eta,\alpha,\lambda,\rho)$ be a structure of an
Ann-category $\mathcal S$ of type $(R,M).$  We define a function
$\sigma:R^4\rightarrow M$ by
\begin{equation}\label{ht4}
\sigma(x,y,z,t)= \xi(x+y,z,t)-\xi(x,y,z)+\eta(y,z)+\xi(x,z,y)-\xi(x+z,y,t)
\end{equation}
This equation shows that $\sigma$ is just the morphism
$$\mathbf{v}:(x+y)+(z+t)\rightarrow (x+z)+(y+t)$$
in an Ann-category of  type $(R,M).$

First, the normalized property of $\sigma$ follows from the ones of
$\xi$ and $\eta$
$$\sigma(0,0,z,t)=\sigma(x,y,0,0)=\sigma(0,y,0,t)=\sigma(x,0,z,0)=\sigma(x,0,0,t)=0.$$
We now show
 that the quadruple $\widehat{h}=(\sigma,\alpha,\lambda,\rho)$ satisfies the
 relations $M_1-M_8,$ and $\widehat{h}$ is therefore a 3-cocycle. The relation $M_1$ is
  just the relation $A_{12}.$ The relations $M_2, M_3, M_4, M_5$ are just $A_{11}, A_{10}, A_9, A_8$, respectively.

According to the coherence theorem in an Ann-category of  type
$(R,M)$ the following Diagrams 2, 3 commute {\scriptsize
\[
\divide \dgARROWLENGTH by 2
\begin{diagram}
\node{r[(x+y)+(z+t)]} \arrow{e,t}{id\tx \bf
v}\arrow{s,l}{\mathcal{L}}
\node{r[(x+z)+(y+t)]} \arrow{s,r}{\mathcal{L}}\\
\node{r(x+y)+r(z+t)} \arrow{s,l}{\mathcal{L}\ts\mathcal{L}}
\node{r(x+z)+r(y+t)}\arrow{s,r}{\mathcal{L}\ts\mathcal{L}}\\
\node{(rx+ry)+(rz+rt)}\arrow{e,t}{\bf v}\node{(rx+rz)+(ry+rt)}
\end{diagram}
\]}
\centerline{\scriptsize Diagram 2}
 {\scriptsize \divide
\dgARROWLENGTH by 2 \setlength{\unitlength}{0.5cm}
\begin{picture}(27,9.0)
\put(0,6.2){$[(r+s)+(u+v)]+[(x+y)+(z+t)]$}
\put(0,3.2){$[(r+u)+(s+v)]+[(x+z)+(y+t)]$}
\put(0,0.2){$[(r+u)+(x+z)]+[(s+v)+(y+t)]$}

\put(13,6.2){$[(r+s)+(x+y)]+[(u+v)+(z+t)]$}
\put(13,3.2){$[(r+x)+(s+y)]+[(u+z)+(v+t)]$}
\put(13,0.2){$[(r+x)+(u+z)]+[(s+y)+(v+t)]$}

\put(10,6.4){\vector(1,0){2.8}} \put(11,6.7){${\bf v}$}
\put(10,0.4){\vector(1,0){2.8}} \put(10.7,0.7){${\bf v}+{\bf v}$}

\put(4.7,5.7){\vector(0,-1){1.6}} \put(3,4.8){${\bf v}+{\bf v}$}

\put(4.7,2.7){\vector(0,-1){1.6}} \put(4.1,1.8){${\bf v}$}

\put(17.8,5.7){\vector(0,-1){1.6}} \put(18.2,4.8){${\bf v}+{\bf v}$}

\put(17.8,2.7){\vector(0,-1){1.6}} \put(18.2,1.8){${\bf v}$}
\end{picture}}
\centerline{\scriptsize Diagram 3\;\;\;\;}
\vspace{0.3cm}

These commutative diagrams imply the relations $M_6, M_8$. The
relation $M_7$ follows from a
 commutative diagram which is analogous to the Diagram 2, where $r$ is tensored on the right side.
\end{proof}

\begin{lemma}\label{bd7}
Each Mac Lane 3-cocycle $(\sigma,\alpha,\lambda,\rho)$ is induced by
a structure $(\xi, \eta,$ $ \alpha, \lambda, \rho)$
 of an Ann-category of type $(R,M).$
\end{lemma}
\begin{proof}
Let $(\sigma,\alpha,\lambda,\rho)$ be an element in $Z^{3}_{MaL}(R,M)).$ Set
$$\xi(x,y,z)=-\sigma(x,y,0,z),\eta(x,y)= \sigma(0,x,y,0),$$
we obtain
a 5-tuple of functions $h=(\xi,\eta,\alpha,\lambda,\rho).$ The
normalized properties of  $\xi, \eta$ follow from that of $\sigma.$

We now show
that $h$ is a structure of an Ann-category of  type
$(R,M).$ First, the relations $A_{12}-A_9$ are just $M_1-M_4$. The
relation $A_1$ follows from $M_8$ when  $u=0=x=y=z.$ The relation
$A_3$ follows from $M_8$ when $r=s=v=0=x=z=t.$ The relations $A_4$
and $A_5$ follow from $M_6$ and $M_7$, respectively,  when $x=t=0.$
The relations $A_6$ and $A_7$ follow from $M_6$ and $M_7$,
respectively, when  $z=0.$

To prove the relation $A_2,$ take $s=u=0=x=z=t$ in $M_8$ we obtain
\begin{equation}\label{ht5}
-\xi(r,y,v)+\xi(r,v,y)-\eta(v,y)+\sigma(r,v,y,0)=0
\end{equation}
Now, take $r=u=0=y=z=t$ in $M_8$ we obtain
$$
-\xi(x,s,v)+\eta(s,x)-\eta(s+v,x)+\sigma(s,v,x,0)=0.$$
In other words,
\begin{equation}
-\xi(y,r,v)+\eta(r,y)-\eta(r+v,y)+\sigma(r,v,y,0)=0 \label{ht6}
\end{equation}
Subtracting  \eqref{ht6} from \eqref{ht5}, we obtain the relation
$A_2.$

Finally, to prove the relation $A_8,$
note that
 $\sigma$ can be presented by
 $\xi, \eta$ as in \eqref{ht4}. Indeed, take $v=0=x=y=z$ in $M_8$ we obtain
\begin{equation}\label{ht7}
\sigma(r,s,u,t)+\xi(r+u,s,t)-\xi(r+s,u,t)-\sigma(r,s,u,0)=0.
\end{equation}
Now, take $v=s, y=u$ in \eqref{ht6} we obtain
\begin{equation}\label{ht8}
\xi(r,u,s)-\xi(r,s,u)-\eta(s,u)+\sigma(r,s,u,0)=0.
\end{equation}
Adding  \eqref{ht7} to \eqref{ht8} and doing  some appropriate calculations, we get \eqref{ht4}.

Because of \eqref{ht4}, $M_8$ becomes $A_8.$ This means the 5-tuple
of functions $h=(\xi,\eta,\alpha,\lambda,\rho)$
 is a structure of an Ann-category of  type $(R,M).$ Further,  this structure induces the
 3-cocycle $\widehat{h}=(\sigma,\alpha,\lambda,\rho).$
\end{proof}
\begin{lemma}\label{bd8}
The structures $h$ and $h'$ of the Ann-category of  type $(R,M)$ are
cohomologous if and only if the corresponding 3-cocycles
$\widehat{h},\widehat{h'}$ are cohomologous.
\end{lemma}
\begin{proof}
By Lemma \ref{bd7},  the structures $h$ and $ h'$ induce elements
$\widehat{h}$ and $ \widehat{h'}$ in $Z^{3}_{MaL}(R,M),$
respectively.
 By Lemma \ref{bd5}, 
the functions $\alpha -\alpha'$, $\lambda - \lambda',$ $\rho -
\rho'$ satisfy the relations $M_{10}-M_{12},$ where $\breve{F}=\tau,
\widetilde{F}=\nu.$ Besides, the following diagram commutes because
of the coherence of a symmetric monoidal functor {\scriptsize
\[
\divide \dgARROWLENGTH by 2
\begin{diagram}
\node{F((x+y)+(z+t))}\arrow{s,l}{F(\bf v)}
\arrow{e,t}{\breve{F}} \node{F(x+y)+F(z+t)} \arrow{s,r}{\breve{F}+\breve{F}}\\
\node{F((x+z)+(y+t))} \arrow{s,l}{\breve{F}}
\node{(F(x)+F(y))+(F(z)+F(t))}\arrow{s,r}{\bf v'}\\
\node{F(x+z)+F(y+t)} \arrow{e,t}{\breve{F}+\breve{F}}
\node{(F(x)+F(z))+(F(y)+F(t)).}
\end{diagram}
\]}
\centerline{\scriptsize Diagram 4}
\vspace{0.2pt}

Note that $F=id$ and $\breve{F}=\tau,$  so the above commutative diagram implies
\begin{eqnarray}
\sigma(x,y,z,t)-\sigma'(x,y,z,t)&=&\tau(x+y,z+t)+\tau(x,y)+\tau(z,t)-\tau(x+z,y+t)\nonumber\\
&&-\tau(x,z)-\tau(y,t).\nonumber
\end{eqnarray}
That means $\sigma-\sigma'$ satisfies $M_9.$ Thus, $\widehat{h}$ and
$\widehat{h'}$ belong to the same cohomology class of
$H^{3}_{{MaL}}(R,M).$

Now, assume that $\widehat{h}-\widehat{h'}\in B^{3}_{MaL}(R,M).$
Then $\alpha -\alpha'$, $\lambda - \lambda',$ $\rho - \rho'$ satisfy
$M_{10}-M_{12}$ which are just the relations $A_{15}-A_{17}$. By
\eqref{ht4}, the definition of $\sigma$ and the normalized property
of $\xi, \eta,$ we have
$$\xi(x, y, z)= -\sigma(x, 0, y, z),\ \xi'(x, y, z)= -\sigma'(x, 0,y, z),$$
$$\eta(x, y)= \sigma(0, x, y, 0),\  \eta'(x, y)= \sigma'(0, x, y, 0).$$
Therefore, 
$A_{13}$, $A_{14}$ are obtained from $M_9$,  and thus $h,h'$ are cohomologous structures.
\end{proof}
Let Struct$[R,M]$ denote the set of cohomology  classes of
structures on $(R,M)$. Then, Lemmas \ref{bd6}, \ref{bd7}, \ref{bd8} lead to the following
result.
\begin{proposition} There exists a bijection
\begin{eqnarray*}
\mathrm{Struct}[R,M]&\rightarrow& H^{3}_{MacL}(R,M) \\
\ [h=(\xi,\eta,\alpha,\lambda,\rho)] &\mapsto&
[\widehat{h}=(\sigma,\alpha,\lambda,\rho)]
\end{eqnarray*}
\end{proposition}
By the above lemma, we regard
each cohomology class
$[h]=[(\xi,\eta,\alpha,\lambda,\rho)]$ as an element of the group
$H^{3}_{MacL}(R,M)$.

Let $\mathbf{Ann}$ refer to the category whose objects are Ann-categories, and whose morphisms are their Ann-functors.

We determine the category $\bf H^{3}_{Ann}$ whose objects are triples $(R,M,[h]),$ where $[h]\in H^3_{MacL}(R,M)$. 
 A morphism $(R,M,[h])\rightarrow (R',M',[h'])$ in $\bf H^{3}_{Ann}$ is a pair $(p, q)$ such that there exists
 a function $g:R^2\rightarrow M'$ so that $(p,q,g):(R,M,h)\rightarrow (R',M',h')$ is an Ann-functor, that is,  $[p^\ast h']
 =[q_\ast h]\in H^3_{MacL}(R,M')$. The composition in $\bf H^{3}_{Ann}$ is defined by
$$(p', q')\circ (p,q)=(p'p, q'q).$$
Note that, {\it Ann-functors $F, F':\mathcal A \rightarrow \mathcal
A'$ are homotopic if and only if $F_i=F_i', i=0,1$ and
$[g_F]=[g_{F'}]$ in $H^{2}_{MacL}(R,M)$}. Denote by
$$\text{Hom}_{(p,q)}^{{Ann}}[\mathcal A,\mathcal A']$$ the set of
homotopy classes of Ann-functors from $\mathcal A$ to $\mathcal A'$
inducing the same pair $(p,q)$, we prove the following
classification result.

\begin{theorem} [Classification Theorem]\label{dl51}
There is a functor
\begin{eqnarray*}
 d:{\bf Ann}&\rightarrow &\bf H^{3}_{Ann}\\
\mathcal A &\mapsto &(\pi_0\mathcal A, \pi_1\mathcal A, [h_{\mathcal A}])
\end{eqnarray*}
which has the following properties:

\emph {i)}   $dF$ is an isomorphism if and only if $F$ is an
 equivalence,

\emph {ii)} $d$ is  surjective on  objects,

\emph {iii)} $d$ is full, but not faithful. For $(p, q):d\mathcal A\rightarrow d\mathcal A',$  there is a bijection
\begin{equation}\label{ht12}
\overline{d}:\mathrm{Hom}_{(p, q)}^{\text{Ann}}[\mathcal A, \mathcal
A']\rightarrow H^2_{MacL}(\pi_0\mathcal A, \pi_1\mathcal
A').
\end{equation}
\end{theorem}\label{dl14}
\begin{proof}
In the Ann-category  $\mathcal A$, for each stick $(X_s,i_X)$ one
can construct a reduced Ann-category $(\pi_0\mathcal A,\pi_1\mathcal
A,h)$. If the choice of the stick is modified, then the 3-cocycle
$h$ changes to a cohomologous 3-cocycle $h'.$ Therefore,
$\mathcal{A}$ uniquely determines   an element $[h]\in
H^3(\pi_0\mathcal{A},\pi_1\mathcal{A})$.

For Ann-functors
$$\mathcal A\stackrel{F}{\longrightarrow}\mathcal A'\stackrel{F'}{\longrightarrow}\mathcal A'',$$
\noindent 
it can be seen that $d(F'\circ F)=dF'\circ dF$, and $d({id_\mathcal
A})=id_{d\mathcal A}$. Therefore, $d$ is a functor.

i) According to Proposition \ref{md22}.

ii) If $(R,M,[h])$ is an object of $\bf{H}^3_{Ann}$, then $\mathcal
S=(R,M,h)$ is an Ann-category of  type $(R,M),$ and obviously
$d\mathcal S=(R,M,[h])$.

iii) If $(p,q)$ is a morphism in
$\text{Hom}_{\bf{H}^3_{\text{Ann}}}(d\mathcal A,d\mathcal A')$, then
there is a function
 $g=(\tau, \nu)$, $\tau, \nu:(\pi_0\mathcal{A})^2\rightarrow \pi_1\mathcal{A'}$ satisfying the relation \eqref{ht9}, and therefore
$$K=(p,q,g):(\pi_0\mathcal A, \pi_1\mathcal A,h_\mathcal A)\rightarrow
(\pi_0\mathcal A', \pi_1\mathcal A',h_{\mathcal A'})$$ is an
Ann-functor. Thus, the composition
 $F=H' K
G:\mathcal A\to\mathcal A'$ is an Ann-functor and $d F=(p,q)$. This
shows that $d$ is full.

In order to obtain the bijection \eqref{ht12}, we prove
that the
correspondence
\begin{align}\label{ht13}
\Omega:\text{Hom}_{(p,q)}^{\text{Ann}}[\mathcal A,\mathcal A']&\rightarrow\
\text{Hom}_{(p,q)}^{\text{Ann}}[S_{\A},S_{\A'}]\\
{[F]}&\mapsto [S_F]\notag
\end{align}
is a bijection.

Clearly, if $F, F':\mathcal A\rightarrow \mathcal A'$ are homotopic
then induced Ann-functors  $S_F, S_{F'}$ are homotopic. Conversely,
if $S_F$ and $S_{F'}$ are homotopic then the compositions
$E=H'(S_F)G$ and $E'=H'(S_{F'})G$ are homotopic. Ann-functors  $E$
and $ E'$ are homotopic to $F$ and $F'$, respectively. So, $F$ and
$F'$ are homotopic. This shows that  $\Omega$ is an injection.

Now, if $K=(p,q,g):S_{\A}\rightarrow S_{\A'}$ is an Ann-functor then the composition
$$F=H' K
G:\mathcal A\rightarrow\mathcal A'$$ is an Ann-functor with $S_F=K$,
that is,
 $\Omega$ is  surjective. Now, the bijection \eqref{ht12} is the composition of \eqref{ht13} and \eqref{ht11}.
\end{proof}

Based on Theorem \ref{dl51}, Ann-categories having the same first
two invariants can be classified  up to equivalence.

Let $R$ be a ring with a unit, $M$ be an $R$-bimodule which is
regarded as a ring with null-multiplication. We say that
 the Ann-category $\mathcal A$ has a {\it pre-stick of  type} $(R,M)$ if there is a pair of ring isomorphisms
  $\epsilon=(p,q)$
\[p: R\to \pi_0\mathcal A,\quad q: M\to \pi_1\mathcal A\]
which are compatible with the module action,
\[q(su)=p(s)q(u),\]
 where $s\in R, u\in M$. The pair $(p, q)$ is called a {\it pre-stick of  type} $(R, M)$ to the Ann-category  $\mathcal A$.

 A {\it morphism} between two Ann-categories $\mathcal A, \mathcal A'$ having pre-sticks of  type $(R, M)$
  (with their  pre-sticks  are $\epsilon=(p, q)$ and $\epsilon'=(p', q')$, respectively) is an Ann-functor
  $(F,\Fa, \Fm): \mathcal A\ri \mathcal A'$ such that the following diagrams commute
 \[\begin{diagram}
\node{\pi_0\mathcal A}\arrow[2]{e,t}{ F_0} \node[2]{\pi_0\mathcal A'}
\node[3]{\pi_1\mathcal A}\arrow[2]{e,t}{F_1}
\node[2]{\pi_1\mathcal A'}\\
\node[2]{R,}\arrow{nw,b}{p}\arrow{ne,b}{p'}
\node[5]{M,}\arrow{nw,b}{q}\arrow{ne,b}{q'}
\end{diagram}\]
where $(F_0, F_1)$ is a pair of homomorphisms induced by
$(F,\Fa,\Fm).$

Clearly, it follows from the definition of an Ann-functor  that
$F_0, F_1$ are isomorphisms, 
therefore $F$ is an equivalence.

Denote by
$${\bf Ann}[R, M]$$
  the set of equivalence classes of Ann-categories whose
pre-sticks are of  type $(R, M)$. One can prove the following result
based on Theorem \ref{dl51}.
\begin {theorem}  There is a bijection
\begin{eqnarray*}
\Gamma:{\bf Ann}[R, M]&\ri& H^{3}_{MacL}(R, M)\\
{[\mathcal A]}&\mapsto&  q^{-1}_\ast p^\ast [h_{\mathcal A}]
\end{eqnarray*}
\end {theorem}
\begin {proof}
By Theorem \ref{dl51}, each Ann-category  $\mathcal A$ determines a
unique element $[h_{\mathcal A}]\in H^3_{MacL}(\pi_0\mathcal A,
\pi_1\mathcal A)$, and hence an element
$$\epsilon[h_{\mathcal A}]=q^{-1}_\ast p^\ast [h_{\mathcal A}]\in
H^3_{MacL}(R, M).$$ Now if $F:\mathcal A \rightarrow \mathcal A'$ is
a functor between Ann-categories whose pre-sticks are of  type
$(p,q),$ then the induced Ann-functor $S_F=(p, q, g_F)$ satisfies
the relation \eqref{ht9}, and therefore
$$p^\ast [h_{\mathcal A'}]=q_\ast [h_{\mathcal A}].$$
One can check that
$$\epsilon'[h_{\mathcal A'}]=\epsilon[h_{\mathcal A}].$$
This means $\Gamma$ is a map. Moreover, it is an injection. Indeed,
if $\Gamma[\mathcal A]=\Gamma[\mathcal A']$, then
$$\epsilon(h_{\mathcal A})-\epsilon'(h_{\mathcal A'})=\partial g.$$
Thus, there exists an Ann-functor $J$ of  type $(id,id)$ from
$\mathcal I=(R, M, \epsilon(h_{\mathcal A}))$ to $\mathcal I'=(R,
M,\epsilon'(h_{\mathcal A'}))$. The composition
$$\mathcal A\stackrel{G}{\longrightarrow}S_{\A}\stackrel{\epsilon^{-1}}{\longrightarrow} \mathcal I\stackrel{J}{\longrightarrow} \mathcal I'\stackrel{\epsilon'}{\longrightarrow}S_{\A'}\stackrel{H'}{\longrightarrow}\mathcal A'$$
shows that $[\mathcal A]=[\mathcal A']$, and $\Gamma$ is an
injection.
Obviously, $\Gamma$ is   surjective.
\end {proof}
In \cite{Q2}, the author proved that each structure of a regular
Ann-category  of  type $(R,M)$ (that is, a structure satisfies the
{\it regular} condition, $\eta(x,x)=0$) is an element in the group
$Z^{3}_{Sh}(R,M)$ of Shukla 3-cocycles. From Classification Theorem
4.4 \cite{Q2} and Theorem \ref{dl14}, 
the following result is obtained.

\begin{corollary} There is an injection
$$H^{3}_{Sh}(R,M)\hookrightarrow H^{3}_{MacL}(R,M).$$
\end{corollary}

\section{Appendix: A categorical ring which is not an Ann-category}
Below, we construct a categorical ring which is not an Ann-category.

Let $R$ be a ring with a unit and $A$ be a $R$-bimodule. Then, one
constructs a categorical ring
  $\mathcal R$ as follows. First,  $\mathcal R$ is a category defined as in Section
 2. The objects of $\mathcal R$ are elements of $R$, the morphisms
 in $\mathcal R$ are automorphisms $(r,a):r\ri r$, $a\in A$.
Composition is the addition on  $A$. Operations
 $\oplus,\otimes$ on $\mathcal R$ are given by
 $$r\oplus s=r+s,\;\;(r,a)\oplus(s,b)=(r+s,a+b),$$
 $$r\otimes s=rs,\;\;(r,a)\otimes (s,b)=(rs,rb+as).$$
Suppose that the system  $(\mathcal R,\oplus,\otimes)$ has a left
distributivity constraint
 $$\lambda_{r,s,t}:r(s+t)\ri rs+rt$$
given by $\lambda_{r,s,t}=(\bullet,\lambda(r,s,t))$, where
 $\lambda:R^3\ri A$, and other constraints are strict. Then, the commutative diagrams in the axioms of
  a categorical ring are equivalent to
 the equations

\vspace{3pt} 
\noindent $R_1.\;r\lambda(s,t,u)-\lambda(rs,t,u)+\lambda(r,st,su)=0$,\\
$R_2.\;\lambda(r,s,t)u-\lambda(r,su,tu)=0$,\\
$R_3.\;\lambda(1,s,t)=0,$\\
$R_4.\;\lambda(r,s+t,u+v)+\lambda(r,s,t)+\lambda(r,u,v)=\lambda(r,s+u,t+v)+\lambda(r,s,u)+\lambda(r,t,v),$\\
$R_5.\;\lambda(r+r',s,t)=\lambda(r,s,t)+\lambda(r',s,t).$

\vspace{3pt}
Let $R$ be the ring of dual numbers on $\mathbb Z$,
$R=\{a+b\epsilon\ |\ a,b\in\mathbb Z,\epsilon^2=0\}$ and $A=\mathbb
Z\cong R/(\epsilon)$. Then, $A$ is a $R$-bimodule with the natural
actions
$$(a+b\epsilon)k=ak=k(a+b\epsilon).$$
The function $\lambda:R^3\ri A$, defined by
$$\lambda(a_r+b_r\epsilon,a_s+b_s\epsilon,a_t+b_t\epsilon)=b_r(a_s+a_t),$$
is satisfies the equations  $R_1-R_5$, so that
$\mathcal R$ is a categorical ring.

It is clear that if  $b_r\neq 0$ and $a_s\neq 0$, then
$\lambda(r,0,s)\neq 0$. Thus, by Theorem \ref{dl3}, $\mathcal R$ is
not an Ann-category.

\vspace{3pt}
One can deduce that:

1. Since the function $\lambda$ is not normalized,
$\widehat{h}=(0,\lambda,0,0)\notin Z^3_{MacL}(R,A)$. This means that
the classification theorem in \cite{JP} is wrong.

2. The   condition  $(U)$ in the following theorem is necessary.

\noindent{\bf Theorem 4} \cite {QHT}. {\it Each categorical ring
$\mathcal R$ is an Ann-category if and only if it satisfies the
following condition.

$(U):$  Each of  pairs $(L^A, \hat L^A)$, $(R^A, \hat R^A)$, $A\in
\mathcal R$, is an $\oplus$-functor which is compatible with the
unit constraint $(0,{\bf g,d})$ with respect  to the operation
$\oplus$.}

{\bf Acknowledgement} The author is much indebted to the referee,
whose useful observations greatly improved my exposition.

\begin{center}

\end{center}
Address: Department of Mathematics\\
\qquad\qquad\ Hanoi National University of Education\\
\qquad\qquad\ 136 Xuan Thuy Street,\\ 
 \qquad\qquad\ Hanoi, Vietnam.\\
Email: cn.nguyenquang@gmail.com
\pagestyle{myheadings}


\begin{thebibliography}{99}
\bibitem {FW}
 A. Fr\"{o}hlich and C. T. C Wall, {\it Graded monoidal categories}, Compositio Math. {\bf 28} (1974). 229-285.
\bibitem {JP}
 M. Jibladze and T. Pirashvili, \emph{ Third Mac Lane cohomology via categorical rings,} J. Homotopy Relat. Struct. {\bf 2} (2007) 187-216.

\bibitem {KV}
 M. M. Kapranov and V. A. Voevodsky, \emph{2-Categories and Zamolodchikov Tetrahedra Equations,} Proceedings of Symposia in Pure Mathematics, Vol. {\bf 56}  (1994), Part 2, 177-259\\
\bibitem{La1}
M. L. Laplaza, \emph{Coherence for distributivity,} Lecture Notes in Math, {\bf 281}  (1972), 29-65.\\
\bibitem{La2}
M. L. Laplaza, \emph{Coherence for Categories with Group Structure: an alternative approach,} J. Algebra, {\bf 84} (1983), 305-323.\\

\bibitem {ML}
 S. Mac Lane, \emph{Extensions and obstruction for rings,} Illinois J. Mathematics, {\bf 2} (1958), 316-345.

 \bibitem {PQT}
  C. T. K. Phung, N. T. Quang, N. T. Thuy, \emph{Relation between Ann-categories and ring categories}, Comm. Korean Math. Soc. {\bf 25}, No 4 (2010), 523-535.

\bibitem {Q1}
 N. T. Quang, \emph{Introduction to Ann-categories,} J. Math. Hanoi, No.15, {\bf 4} (1987), 14-24. arXiv:math.CT/0702588v2 21 Feb 2007.\\
\bibitem {Q2}
 N. T. Quang, \emph{Structure of Ann-categories and Mac Lane-Shukla cohomology,}  East-West J.  Mathematics, Vol {\bf 5} , No 1 (2003), 51-66.\\

\bibitem {QHT}
 N. T. Quang, D. D. Hanh and N. T. Thuy, \emph{On the Axiomatics of Ann-categories,} JP Journal of Algebra, Number Theory and Applications, Vol {\bf 11} , No 1 (2008), 59-72.\\
\bibitem {QH1} N.  T.  Quang, D. D. Hanh, \emph{Homological classification of Ann-functors}, {East-West J. of Mathematics}, Vol {\bf 11}, No 2 (2009),  195-210.




\bibitem {SR}
 N. Saavedra Rivano, {\it Cat\'{e}gories Tannakiennes,} Lecture Notes in Mathematics, Vol. {\bf 265},  Springer-Verlag Berlin and New York (1972). \\
\bibitem {Sh}
 U. Shukla, \emph{Cohomologie des algebras associatives,} Ann. Sci. Ecole Norm. Sup., {\bf 78} No 2 (1961), 163-209.\\

\bibitem {Si}
H. X. Sinh, \emph{Gr-cat\'{e}gories}, Universit\'{e} Paris VII,
Th\'{e}se de doctorat (1975).

\end{thebibliography}
\end{document}